\documentclass[12pt]{article}

\usepackage[margin = 2.90cm]{geometry}
\usepackage{setspace}
\usepackage{amsmath,amsfonts,amssymb}
\usepackage{amsthm}
\usepackage{mathrsfs}
\usepackage{enumerate}
\usepackage{hyperref}
\usepackage{pdfsync}
\usepackage{pdfpages}
\usepackage{graphicx}
\usepackage{float}
\usepackage{cleveref}

\newcommand{\N}{\ensuremath{\mathbb{N}}}

\newcommand{\R}{\ensuremath{\mathbb{R}}}

\newcommand{\Ce}{\ensuremath{\mathscr{C}}}

\theoremstyle{plain}           	
\newtheorem{theorem}{Theorem}
\newtheorem{lemma}[theorem]{Lemma}
\newtheorem{corollary}[theorem]{Corollary}

\theoremstyle{remark}
\newtheorem{remark}[theorem]{Remark}

\DeclareMathOperator*{\tr}{tr}
\DeclareMathOperator*{\rk}{rk}
\newcommand{\bangle}[1]{\left\langle #1 \right\rangle}
\newcommand{\inprod}[2]{\bangle{#1, #2}}

\title{Small codes}
\date{}
\author{Igor Balla}

\begin{document}

\maketitle 

\begin{abstract}
Determining the maximum number of unit vectors in $\R^r$ with no pairwise inner product exceeding $\alpha$ is a fundamental problem in geometry and coding theory. In 1955, Rankin resolved this problem for all $\alpha \leq 0$ and in this paper, we show that the maximum is $(2+o(1))r$ for all $0 \leq \alpha \ll r^{-2/3}$, answering a question of Bukh and Cox. Moreover, the exponent $-2/3$ is best possible. As a consequence, we conclude that when $j \ll r^{1/3}$, a $q$-ary code with block length $r$ and distance $(1-1/q)r - j$ has size at most $(2 + o(1))(q-1)r$, which is tight up to the multiplicative factor $2(1 - 1/q) + o(1)$ for any prime power $q$ and infinitely many $r$. When $q = 2$, this resolves a conjecture of Tiet\"av\"ainen from 1980 in a strong form and the exponent $1/3$ is best possible. Finally, using a recently discovered connection to $q$-ary codes, we obtain analogous results for set-coloring Ramsey numbers.
\end{abstract}

\section{Introduction}

Let $S^{r-1}$ denote the unit sphere in $\R^r$ and define
\[
\rho(r, n) =  \min_{v_1, \ldots, v_{n} \in S^{r-1}}{ \max_{i \neq j}{\inprod{v_i}{v_j}}}.
\]
The question of determining $\rho$ (as well as the extremal configurations) was first posed in 1930 by the botanist Tammes \cite{T30} in the context of studying the arrangements of pores of pollen grains. Since then, it has come to be known as an important problem in geometry and coding theory, see e.g. \cite{CS13, KL78, MT15} for more information. Indeed, this question is equivalent to that of packing spherical caps on a unit sphere (when the angular radius of each cap is $\pi/6$, this is equivalent to the kissing number problem) and it is strongly related to the classical problem of packing spheres in a Euclidean space.

In 1955, Rankin \cite{R55} determined $\rho(r, n)$ for all $n \leq 2r$. It is therefore natural to ask what happens when $n$ is slightly larger than $2r$. This question was considered more recently by Bukh and Cox \cite{BC20}, who showed that if $r$ is sufficiently large relative to $k$, then $\rho(r,2r + k) \leq O\left( \sqrt{k} / r \right)$ and they also noted that the linear programming method of Delsarte, Goethals, and Seidel \cite{DGS77} yields the lower bound $\rho(r, 2r + k) \geq (1-o(1)) k / r^2$. In this paper, we substantially improve this lower bound and as a consequence, we determine that $\rho(r, 2r + k) = \Theta(1/r)$ for any fixed $k \in \N$ as $r \rightarrow \infty$.

\begin{theorem} \label{main}
For all $k, r \in \N$, we have $\rho(r, 2r + k) \geq \frac{ \left( \frac{8}{27}k + 1 \right)^{1/3} - 1}{2r+k}$.
\end{theorem}

In general, a collection of unit vectors in a Euclidean space is called a \emph{spherical $L$-code} if all pairwise inner products lie in the set $L$. If we define $M(r, \alpha)$ to be the maximum size of a spherical $[-1,\alpha]$-code in $\R^r$, then it is not hard to see that determining $M$ is equivalent to determining $\rho$. Since it is well known that $M(r, 0) = 2r$, we obtain the following immediate corollary of \Cref{main}.

\begin{corollary} \label{cor_spherical}
For all $r \in \N$ and $\alpha \geq 0$ such that $\alpha = o\left(r^{-2/3}\right)$ as $r \rightarrow \infty$, we have
\[
M(r, \alpha) = (2 + o(1))r.
\]
\end{corollary}

Since $q$-ary codes give rise to spherical codes, we will show that \Cref{cor_spherical} implies a bound on $q$-ary codes. More precisely, a subset $C \subseteq [q]^r$ is called a \emph{$q$-ary code with block length $r$ and distance $s$} if any distinct $x, y \in C$ differ in at least $s$ coordinates and we let $A_q(r,s)$ denote the maximum size of such a code. Similarly to the spherical case, $A_q(r, s)$ is fairly well understood when $s \geq (1 - 1/q)r$ due to the bounds of Plotkin \cite{P60} for binary codes ($q = 2$) and of Blake and Mullin \cite{BM76}, Mackenzie and Seberry \cite{MS88} for $q \geq 3$. In particular, Plotkin showed that 
\begin{equation} \label{eq_plotkin}
A_2(r, r/2) \leq 2r \text{ with equality if there exists a Hadamard matrix of order } r
\end{equation}
and more generally, Mackenzie and Seberry showed that for $r \geq q \geq 3$,
\begin{equation} \label{eq_ms}
A_q(r,(1-1/q)r) \leq q r \text{ with equality if } r, q \text{ are powers of the same prime.}
\end{equation}
Moreover, for binary codes and when $s$ is slightly smaller than $r/2$, i.e.\ when $s = r/2 - j$ for $j = o\left(r^{1/3}\right)$, Tiet\"av\"ainen \cite{T80} conjectured that $A_2(r, r/2 - j) = O(r)$. We will show that any $q$-ary code with block length $r$ and distance $(1- 1/q)r - j$ can be transformed into a spherical $\left[-1, \frac{q j}{(q-1) r} \right]$-code in $\R^{(q-1)r}$, so that applying \Cref{cor_spherical} will yield the following result, which verifies Teit\"av\"ainen's conjecture in a strong form.

\begin{theorem} \label{q-ary}
Let $q \geq 2$ be a fixed integer and let $r \in \N$ and  $j \geq 0$ be such that $(1-1/q)r - j \in \N$. If $j = o\left(r^{1/3}\right)$ as $r \rightarrow \infty$, then
\[
A_q\left( r, \left( 1 - 1/q \right) r - j \right) \leq \left(2 + o(1)\right) \left(q - 1 \right) r.
\]
\end{theorem}

\begin{remark} \label{sidelnikov}
Note that Sidel'nikov \cite{S71} constructed a binary code with block length $r$ and distance $r/2 - \Theta\left(r^{1/3}\right)$ having size $\Theta\left(r^{4/3}\right)$ for infinitely many $r$, so that the exponent $1/3$ appearing in \Cref{q-ary} is best possible for $q = 2$. Moreover, by changing each $0$ to a $1$, each $1$ to a $-1$, and normalizing the resulting vectors, his construction yields a spherical $\left[ -1, \Theta\left(r^{-2/3}\right) \right]$-code of size $\Theta\left(r^{4/3} \right)$ in $\R^r$, so that the exponent $-2/3$ appearing in \Cref{cor_spherical} is also best possible. Furthermore, if $k = \Theta(r)$ then by considering any subset of $2r + k$ vectors from this spherical code, one can see that \Cref{main} is tight up to a multiplicative constant, i.e.\ $\rho(r , 2r + k) = \Theta\left( \frac{k^{1/3}}{2r+k}\right)$.
\end{remark}

It follows directly from the lower bounds of \eqref{eq_plotkin} and \eqref{eq_ms} that for any prime power $q$, the upper bound given in \Cref{q-ary} is tight up to the multiplicative factor $2(1 - 1/q) + o(1)$ for infinitely many $r$. In particular, if there exists a Hadamard matrix of order $r$ and $j = o(r^{1/3})$, then we conclude that $A_2\left( r, r/2 - j \right) = \left(2 + o(1) \right) r$ as $r \rightarrow \infty$.

Very recently, Conlon, Fox, He, Mubayi, Suk, and Verstra\"ete \cite{CFHMSV22} established a connection between error-correcting codes and a variant of the classical Ramsey numbers. For all $q, r, s \in \N$ with $r > s$, the \emph{set-coloring Ramsey number} $R(q; r, s)$ is defined to be the minimum $n$ such that if each edge of the complete graph $K_n$ receives $s$ colors out of a universe of $r$ colors, then there exist a set of $q$ vertices all of whose edges received the same color. In particular, $R(q; r, 1)$ is just the usual $r$-color Ramsey number of a $q$-clique. For all $q, r, s \in \N$ with $r > s$, it was shown in \cite{CFHMSV22} that 
\begin{equation} \label{eq_lower}
R(q+1; r, s) \geq A_q(r,s) + 1
\end{equation}
and on the other hand, Conlon, Fox, Pham, and Zhao \cite{CFPZ23} showed that this bound is approximately tight when $s$ is near $(1-1/q) r$ by proving that for any $\varepsilon > 0$, there exists a $c > 0$ such that if $r, s \in \N$ with $s \leq (1-1/q)r$ and $j = (1-1/q)r - s + 1$, then
\begin{equation} \label{eq_upper}
R(q+1; r, s) \leq \max((1+\varepsilon)A(r, s - c j), \varepsilon s).
\end{equation}

As a consequence of \eqref{eq_upper} and \Cref{q-ary}, we immediately obtain the following upper bound on set-coloring Ramsey numbers when $s$ is slightly below $(1-1/q)r$. 
\begin{corollary} \label{cor_ramsey}
Let $q \geq 2$ be a fixed integer and let $r \in \N$ and  $j \geq 0$ be such that $(1-1/q)r - j \in \N$. If $j = o\left(r^{1/3}\right)$ as $r \rightarrow \infty$, then
\[ 
R(q+1; r, (1-1/q)r - j) \leq (2 + o(1))(q - 1)r.
\]
\end{corollary}

Using \eqref{eq_lower}, together with the lower bound constructions mentioned in \eqref{eq_plotkin}, \eqref{eq_ms}, and \Cref{sidelnikov}, we analogously have that for any prime power $q$, \Cref{cor_ramsey} is tight up to the multiplicative factor $2(1 - 1/q) + o(1)$ for infinitely many $r$ and when $q = 2$, the exponent $1/3$ is best possible. In particular, if there exists a Hadamard matrix of order $r$ and $j = o(r^{1/3})$, then $R(3; r, r/2 - j) = (2 + o(1))r$  as $r \rightarrow \infty$.

\section{Proofs}

In order to establish \Cref{main}, we will need the following simple but powerful lemma, which provides a quantitative version of the fact that a matrix with large trace but whose square has small trace, must have large rank (see e.g.\ \cite{A09} for other applications).

\begin{lemma} \label{Schnirelman}
Let $M$ be a symmetric real matrix with rank $r$. Then 
\[
\tr(M)^2 \leq r \tr(M^2).
\]
\end{lemma}
\begin{proof}
Since $M$ is a symmetric real matrix, $M$ has precisely $r$ non-zero real eigenvalues $\lambda_1,\dots,\lambda_{r}$. Applying Cauchy--Schwarz then yields the desired bound
\[ 
\tr(M)^2 = \left(\sum_{i=1}^r \lambda_i \right)^2 \leq r \sum_{i=1}^r \lambda_i^2 = r \tr(M^2). \qedhere
\]
\end{proof}

Now let $k, r \in \N, \alpha \in [0,1)$, and  let $\Ce$ denote a set of $n = 2r + k$ unit vectors in $\R^d$. We define $M = M_{\Ce}$ to be the corresponding \emph{Gram matrix}, i.e.\ the $\Ce \times \Ce$ matrix satisfying $M(u,v) = \inprod{u}{v}$ for $u, v \in \Ce$. We would like to apply \Cref{Schnirelman} to $M$ and so we will need to bound $\tr(M^2)$, i.e.\ the sum of the squares of the entries of $M$. If $\Ce$ were a spherical $[-\alpha, \alpha]$-code, then we would have $\tr(M^2) \leq n + \alpha^2 n(n-1)$ and so \Cref{Schnirelman} would yield the lower bound $\alpha \geq \Omega\left( k^{1/2} / r \right)$. However, the same bound doesn't follow if all we know is that $\Ce$ is a spherical $[-1,\alpha]$-code, since inner products can be very negative. Nonetheless, the following lemmas show that we can effectively bound the sum of the squares of the negative inner products by using the fact that $M$ is positive semidefinite.

To this end, we will need the following definitions. We think of $M$ as a signed weighted complete graph with vertex set $\Ce$ and so we say that the edge $uv$ is \emph{negative} if $\inprod{u}{v} < 0$. Moreover, for any $u \in \Ce$, we define $N^{+}(u) = \{ v \in \Ce: \inprod{u}{v} \geq 0 \}$ and $N^{-}(u) = \{ v \in \Ce : \inprod{u}{v} < 0\}$. We also define $\gamma(u) = \sum_{v \in N^{-}(u)}{\inprod{u}{v}}$, i.e.\ the sum of the negative edges incident to $u$.

\begin{lemma} \label{beta_edges}
Let $r \in \N, \alpha \in [0,1)$ and let $\Ce$ be a spherical $[-1,\alpha]$-code. For all $u \in \Ce$, we have
\[ \sum_{v \in N^{-}(u)}{\inprod{u}{v}^2} \leq 1 + \alpha \gamma(u)^2 . \]
\end{lemma}

\begin{proof}
Let $M$ be the Gram matrix for $\{u\} \cup N^{-}(u)$ and let $x : \{u\} \cup N^{-}(u) \rightarrow \R$ be the vector with $x(u) = 1$ and $x(v) = - \inprod{u}{v}$ for all $v \in N^{-}(u)$. Since $M$ is positive semidefinite, we obtain
\begin{align*}
0 \leq \inprod{Mx}{x} 
&= 1 + \sum_{v \in N^{-}(u)}{x(v)^2} - 2 \sum_{v \in N^{-}(u)}{x(v)^2} + \sum_{\substack{v, w \in N^{-}(u) \\ v \neq w}}{x(v) x(w) \inprod{v}{w}}\\
&\leq 1 - \sum_{v \in N^{-}(u)}{x(v)^2} + \alpha \sum_{\substack{v, w \in N^{-}(u) \\ v \neq w}}{x(v) x(w)}\\
&\leq 1 - \sum_{v \in N^{-}(u)}{x(v)^2} + \alpha \left( \sum_{v \in N^{-}(u)}{x(v)} \right)^2,
\end{align*}
which is equivalent to the desired result.
\end{proof}

\begin{lemma} \label{gamma_bound}
Let $r \in \N, \alpha \in [0,1)$ and let $\Ce$ be a spherical $[-1,\alpha]$-code with $n = |\Ce|$. Then we have
\[ \sum_{u \in \Ce}{\gamma(u)^2} \leq \frac{27}{4} (1+\alpha n)^2 n .\]
\end{lemma}

\begin{proof}
Let $L = \frac{3}{2}(1 + \alpha n)$ and let $M = M_{\Ce}$ be the Gram matrix of $\Ce$. Also define $B = \{ u \in \Ce : -\gamma(u) \geq L \}$ and let $x: \Ce \rightarrow \R$ be the vector defined by
\[
x(u) = 
\begin{cases} 
      -\gamma(u) & \text{if } u \in B\\
      L & \text{if } u \in \Ce \backslash B. 
\end{cases}
\]
Since $M$ is positive semidefinite, we obtain
\begin{align*}
0 \leq \inprod{Mx}{x}
&= \sum_{u \in \Ce}{x(u)^2} + \sum_{u \in \Ce}{ x(u) \sum_{v \in N^{+}(u) }{x(v) \inprod{u}{v}} } + \sum_{u \in \Ce}{ x(u) \sum_{v \in N^{-}(u) }{x(v) \inprod{u}{v}} } \\
&\leq \sum_{u \in \Ce}{x(u)^2} + \alpha \sum_{u \in \Ce}{ x(u) \sum_{v \in N^{+}(u) }{x(v)} } + L \sum_{u \in B}{ x(u) \sum_{v \in N^{-}(u) }{\inprod{u}{v}} } \\
&\leq \sum_{u \in \Ce}{x(u)^2} + \alpha \left( \sum_{u \in \Ce}{ x(u) } \right)^2 - L \sum_{u \in B}{ x(u)^2 }.
\end{align*}
By Cauchy--Schwarz, we have $\left( \sum_{u \in \Ce}{ x(u) } \right)^2 \leq n \sum_{u \in \Ce}{x(u)^2}$ and thus
\[ 
0 \leq (1 + \alpha n) \sum_{u \in \Ce}{x(u)^2} - L \sum_{u \in B}{ x(u)^2 } = (1 + \alpha n) \sum_{u \in \Ce \backslash B}{x(u)^2} - \frac{1}{2}(1 + \alpha n) \sum_{u \in B}{ x(u)^2 },
\]
so that we conclude $\sum_{u \in B}{ x(u)^2 } \leq 2 \sum_{u \in \Ce \backslash B}{ x(u)^2 } = 2 L^2 n$. Since $\gamma(u)^2 \leq L^2$ for all $u \in \Ce \backslash B$, we obtain the desired bound
\[ 
\sum_{u \in \Ce}{ \gamma(u)^2 } = \sum_{u \in B}{ x(u)^2 } + \sum_{u \in \Ce \backslash B}{ \gamma(u)^2 } \leq 2 L^2 n + L^2n = \frac{27}{4} (1+\alpha n)^2 n. \qedhere
\]
\end{proof}

\begin{proof}[Proof of \Cref{main}]
Let $\Ce$ be a set of $n = 2r + k$ unit vectors in $\R^r$ having maximum pairwise inner product $\alpha = \rho(r, 2r + k)$. Let $M = M_{\Ce}$ be the Gram matrix of $\Ce$. We would like to apply \Cref{Schnirelman} to $M$. To this end, note that $\rk(M) \leq r$ and $\tr(M) = n$. Moreover,
using \Cref{beta_edges} and \Cref{gamma_bound} we have
\begin{align*}
\tr(M^2) 
&= \sum_{u \in \Ce}{\left( \inprod{u}{u}^2 + \sum_{v \in N^{+}(u)}{\inprod{u}{v}^2} +  \sum_{v \in N^{-}(u)}{\inprod{u}{v}^2} \right) }\\
&\leq \sum_{u \in \Ce}{\left( 1 + \alpha^2 n + 1 + \alpha \gamma(u)^2 \right) }\\
&\leq 2n + (\alpha n)^2 + \frac{27}{4}(1+\alpha n)^2 \alpha n.
\end{align*}
Noting that $(\alpha n)^2 + \frac{27}{4}(1+\alpha n)^2 \alpha n \leq  \frac{27}{4} \left( (1 + \alpha n)^3 -1 \right)$, we now apply \Cref{Schnirelman} to conclude 
\begin{align*}
2n + 2k \leq \frac{n^2}{r} = \frac{\tr(M)^2}{r} \leq \tr(M^2) \leq 2n + \frac{27}{4} \left( (1 + \alpha n)^3 -1 \right),
\end{align*}
so that $k \leq \frac{27}{8} \left( (1 + \alpha n)^3 -1 \right)$, which is equivalent to the desired bound.
\end{proof}

Having verified \Cref{main}, we conclude this section by proving \Cref{q-ary}.
\begin{proof}[Proof of \Cref{q-ary}]
Let $u_1,\ldots, u_{q}$ be equidistant unit vectors in $\R^{q-1}$, so that $\inprod{u_i}{u_j} = -1/(q-1)$ for all $i \neq j$. For any $x \in [q]^r$, we let $f(x) \in \R^{(q-1)r}$ be the vector obtained from $x$ by replacing each coordinate $x_i$ with the vector $u_{x_i}$, i.e.\ $f(x)$ is the concatenation of $u_{x_1}, \ldots, u_{x_r}$. Now observe that for all $x, y \in [q]^r$, we have $\inprod{f(x)}{f(y)} = \sum_{i=1}^{r}{\inprod{u_{x_i}}{u_{y_i}}}$, so that $||f(x)||^2 = r$ and moreover, if we let $d(x,y)$ denote the Hamming distance between $x$ and $y$, i.e.\ the number of coordinates in which they differ, then we conclude that
\[
\inprod{f(x)}{f(y)} = r - \frac{q}{q-1} d(x,y).
\]
It follows that for any $q$-ary code $C$ with block length $r$ and distance $(1-1/q)r - j$, the collection $\Ce = \{ f(x) / \sqrt{r} : x \in C\} $ is a spherical $\left[-1, \frac{q j}{(q-1)r} \right]$-code in $\R^{(q-1)r}$, and thus \Cref{cor_spherical} yields $|C| = |\Ce| \leq (2 + o(1))(q-1)r$, as desired.
\end{proof}

\section{Concluding remarks}
Note that if $k \rightarrow \infty$ and $r$ is sufficiently large relative to $k$, then \Cref{main} and the construction of Bukh and Cox \cite{BC20} yield the bounds
\[
\Omega\left(k^{1/3}\right) \leq r \cdot \rho(r, 2r + k) \leq O\left(k^{1/2}\right),
\]
so it would be interesting to determine the order of growth of $r \cdot \rho(r, 2r + k)$ when $r$ is sufficiently large relative to $k$. Since we know that \Cref{main} is tight up to a multiplicative constant when $k = \Theta(r)$ (see \Cref{sidelnikov}), we conjecture that the answer should be $\Theta\left(k^{1/3}\right)$.

If we let $j = j(r)$ be any function of $r$ which satisfies $0 \leq j \ll r^{1/3}$ as $r \rightarrow \infty$, then \Cref{q-ary} and the lower bounds of \eqref{eq_plotkin} and \eqref{eq_ms} imply that for any prime power $q$, the quantity
\[
B_q(j) = \limsup_{r \rightarrow \infty} \frac{1}{r} A_q\left( r, \left( 1 - 1/q \right) r - j \right)  
\] 
satisfies $q \leq B_q(j) \leq 2(q - 1)$. In particular, we have shown that $B_2(j) =  2$, so it would be interesting to see if our methods can be extended in order to determine $B_q(j)$ for $q \geq 3$. In view of \eqref{eq_lower} and \eqref{eq_upper}, this would yield corresponding results for set-coloring Ramsey numbers.

\end{document}